\documentclass[a4paper,reqno,12pt]{amsart}
\usepackage{verbatim}
\usepackage{amssymb,amsmath,amsthm}
\usepackage{amsfonts}
\usepackage{array}

\newtheorem{theorem}{Theorem}[section]
\newtheorem{lemma}{Lemma}[section]
\newtheorem{corollary}{Corollary}[section]
\newtheorem{conjecture}{Conjecture}[section]

\theoremstyle{remark}

\newcommand{\qbin}[2]{\genfrac{[}{]}{0pt}{}{#1}{#2}_q}
\newcommand{\qq}[1]{[#1]_q}

\begin{document}

\setcounter{page}{1}

\title[]{Some $q$-congruences for homogeneous and quasi-homogeneous multiple $q$-harmonic sums}

\author{Kh.~Hessami Pilehrood}
\address{The Fields Institute, 222 College Street, Toronto, Ontario M5T 3J1, Canada}
\email{hessamik@gmail.com}

\author{T.~Hessami Pilehrood}
\address{The Fields Institute, 222 College Street, Toronto, Ontario M5T 3J1, Canada}
\email{hessamit@gmail.com}

\author{R.~Tauraso}
\address{Dipartimento di Matematica, % \\
Universit\`a di Roma ``Tor Vergata'', % \\
via della Ricerca Scientifica, %\\
00133 Roma, Italy}
\email{tauraso@mat.uniroma2.it}


%\date{}

\begin{abstract}
We show some new Wolstenholme type $q$-congruences for some classes of multiple $q$-harmonic sums of arbitrary depth with strings of indices composed of ones, twos and threes. Most of these results are $q$-extensions of the corresponding congruences for ordinary multiple harmonic sums obtained by the authors in a previous paper. Finally, we pose a conjecture concerning two kinds of cyclic sums of multiple $q$-harmonic sums.
\end{abstract}

\maketitle

\section{Introduction}

For two $l$-tuples of non-negative integers
${\bf s}:=(s_1, \ldots, s_l)$ and ${\bf t}:=(t_1,\ldots, t_l)$ and a non-negative integer $n,$ we define two classes of multiple $q$-harmonic sums
$$
H_n^q({\bf s};{\bf t})=H_n^q(s_1,\ldots,s_l;t_1,\ldots,t_l)=\sum_{1\le k_1<\cdots<k_l\le n}\frac{q^{k_1t_1+\cdots+k_lt_l}}{[k_1]_q^{s_1}\cdots[k_l]_q^{s_l}},
$$
$$
S_n^q({\bf s};{\bf t})=S_n^q(s_1,\ldots,s_l;t_1,\ldots,t_l)=\sum_{1\le k_1\le\cdots\le k_l\le n}\frac{q^{k_1t_1+\cdots+k_lt_l}}{[k_1]_q^{s_1}\cdots[k_l]_q^{s_l}},
$$
where
$$
\qq{n}=\frac{1-q^n}{1-q}=1+q+\dots+q^{n-1}
$$
is a $q$-analog of the non-negative integer $n$,
and put
$$
H_n^q({\bf s}):=H_n^q(s_1,\ldots,s_l;s_1-1,\ldots,s_l-1), \,\,\, S_n^q({\bf s}):=S_n^q(s_1,\ldots,s_l;s_1-1,\ldots,s_l-1).
$$
The number $l({\bf s}):=l$ is called the depth (or length) and $w({\bf s}):=\sum_{j=1}^ls_j$ is the weight of the multiple harmonic sum.
By convention, we put $H_n^q({\bf s};{\bf t})=S_n^q({\bf s};{\bf t})=0$ if $n<l$, and $H_n^q(\emptyset)=S_n^q(\emptyset)=1$.

Congruences for ordinary single and multiple harmonic sums have been studied since the  nineteenth century
(see \cite{G900, HPT14, Ho04, Le38, W862, Zh08} and references therein). It is well known that many multiple harmonic sums modulo a prime $p$ (or a power of $p$)
can be expressed in terms of Bernoulli numbers. The situation with $q$-analogs is much less known.
The first $q$-congruences for $q$-analogs of harmonic numbers, namely the $q$-analogs of Wolstenholme's theorem, were obtained by Andrews \cite{An99}
who proved that
for all primes $p\ge 3,$
$$
H_{p-1}^q(1;0)=\sum_{j=1}^{p-1}\frac{1}{[j]_q}\equiv \frac{p-1}{2} (1-q) \pmod{[p]_q},
$$
as well as
\begin{equation} \label{qcon-3}
H_{p-1}^q(1;1)=\sum_{j=1}^{p-1}\frac{q^j}{[j]_q}\equiv -\frac{p-1}{2}(1-q) \pmod{[p]_q}.
\end{equation}
Then Shi and Pan \cite{SP07} proved that for primes $p\ge 5$,
\begin{equation*}
\begin{split}
H_{p-1}^q(2;0)&\equiv -\frac{(p-1)(p-5)}{12} (1-q)^2 \pmod{[p]_q}, \\[3pt]
H_{p-1}^q(2;1)&\equiv -\frac{p^2-1}{12} (1-q)^2 \pmod{[p]_q}.
\end{split}
\end{equation*}
Here the above congruences are considered over the ring of polynomials with integer coefficients in variable $q.$
It is clear that $[p]_q$, as the $p$th cyclotomic polynomial, is irreducible over $\mathbb{Q}$ and therefore, the denominators of the rational functions above are coprime to $[p]_q.$

Dilcher \cite{Di08} was the first who noticed that the higher order $q$-harmonic numbers can be expressed modulo $[p]_q$ in terms of the degenerate Bernoulli numbers,
namely in terms of the sequence $K_n(p)$ defined by the generating function
\begin{equation} \label{qcon-1}
\frac{p(z-1)}{z^p-1}=\sum_{n=0}^{\infty}(-1)^{n-1}K_n(p)(z-1)^n \qquad\quad\mbox{for $|z-1|<1$.}
\end{equation}
He proved (see \cite{Di08} and also \cite{Zh13}) that for all positive integers $n$ and primes $p\ge 3,$
\begin{equation} \label{qcon-4}
H_{p-1}^q(n;1)\equiv K_n(p) (1-q)^n \pmod{[p]_q}
\end{equation}
and
\begin{equation}\label{qcon-4a}
H_{p-1}^q(n;0)\equiv (1-q)^n\Bigl(\frac{p-1}{2}+\sum_{m=2}^nK_m(p)\Bigr) \pmod{[p]_q}.
\end{equation}
The congruences above easily imply (see \cite[Theorem 1.2]{Zh13} and \cite[Theorem 2.1]{Ta13} for a similar result) that for all primes $p>3$ and integers $n\ge t\ge 1$,
\begin{equation} \label{HS}
H_{p-1}^q(n;t)\equiv (1-q)^n\sum_{m=0}^{t-1}\binom{t-1}{m}(-1)^mK_{n-m}(p) \pmod{[p]_q}.
\end{equation}
These results were further extended by Zhao \cite{Zh13} who studied generating functions of multiple $q$-harmonic sums with repeated arguments $H_{p-1}^q(\{s\}^t; \{0\}^t)$, where $\{s\}^t$ denotes $t\ge 0$ consecutive copies of the letter $s$.
In particular, he showed \cite[Corollaries 2.2, 2.3]{Zh13} that
\begin{equation} \label{111}
H_{p-1}^q(\{1\}^t)=H_{p-1}^q(\{1\}^t;\{0\}^t)\equiv \binom{p-1}{t}\frac{(1-q)^t}{t+1} \pmod{[p]_q}
\end{equation}
and
\begin{equation} \label{222}
H_{p-1}^q(\{2\}^t;\{0\}^t)\equiv (-1)^t\frac{2\cdot t!}{(2t+2)!}\binom{p-1}{t}\cdot F_{2,t}(p)\cdot(1-q)^{2t} \pmod{[p]_q},
\end{equation}
where $F_{2,t}(p)$ is a monic polynomial in $p$ of degree $t$ and $p>t$.

In this paper, we prove some new $q$-congruences modulo
$[p]_q$ for further classes of multiple $q$-harmonic sums of arbitrary depth with strings of indices composed of ones, twos and threes.
Some of these results are $q$-extensions of the corresponding congruences for ordinary multiple harmonic sums obtained by the authors in \cite{HPT14}.

We will refer to multiple $q$-harmonic sums on strings $(s,s,\ldots,s)$ as {\sl homogeneous} sums. If the vector $(s,s,\ldots,s)$
is modified by insertion of an element $u$,  we say that $H_n^q(\{s\}^a,u,\{s\}^b)$ (correspondingly, $S_n^q(\{s\}^a,u,\{s\}^b)$)
is a {\sl quasi-homogeneous} multiple $q$-harmonic (non-strict) sum.
We may summarize our results as follows.

\noindent Homogeneous sums: for any integer $t\ge 1$ and any prime $p>t$ we have modulo $[p]_q$,
\begin{align*}
S_{p-1}^q(\{1\}^t)&\equiv -K_t(p)(1-q)^t,\\
H_{p-1}^q(\{2\}^t)&\equiv (-1)^t\binom{p+t}{2t+1} \frac{(1-q)^{2t}}{p(t+1)},\\
S^q_{p-1}(\{2\}^t)&\equiv
(1-q)^{2t}\sum_{i=0}^t
\left(t\binom{t-1}{i}+(t-1)\binom{t}{i}\right)
(-1)^{i}K_{2t-i}(p), \\
H_{p-1}^q(\{3\}^t)&\equiv
\frac{(1-q)^{3t}}{(t+1)p^2}
\left(\binom{p+2t+1}{3t+2}+\binom{-p+2t+1}{3t+2}\right).
 \end{align*}
Quasi-homogeneous sums: for any non-negative integers $a,b$ and any prime $p>2t+1$ with  $t=a+b+1$, we have modulo $[p]_q$,
\begin{align*}
H^q_{p-1}(\{1\}^{a},2,\{1\}^{b})
+H^q_{p-1}(\{1\}^{b},2,\{1\}^{a})&\equiv
-\binom{p+1}{t+2}\frac{(1-q)^{t+1}}{p},\\
S_{p-1}^q(\{1\}^{a},2,\{1\}^{b})+S_{p-1}^q(\{1\}^{b},2,\{1\}^{a})
&\equiv (1-q)^{t+1}(K_{t}(p)-K_{t+1}(p)-K_{a+1}(p)K_{b+1}(p)),\\
H_{p-1}^q(\{2\}^a,3,\{2\}^b)+H_{p-1}^q(\{2\}^b,3,\{2\}^a)
&\equiv
(-1)^{t-1}\binom{p+t}{2t+1}\frac{(1-q)^{2t+1}}{p(t+1)},\\
S_{p-1}^q(\{2\}^a,3,\{2\}^b)+S_{p-1}^q(\{2\}^b,3,\{2\}^a)
&\\
\equiv(1-q)^{2t+1}\sum_{i=0}^{t}
&\left(t\binom{t-1}{i}+(t-1)\binom{t}{i}\right)
(-1)^{i+1}K_{2t-i}(p).
\end{align*}
Our derivations for homogeneous strict sums are based on application of generating functions and properties of cyclotomic and Chebyshev polynomials,
while the congruences for non-strict sums follow from some $q$-identities proved in \cite{HP13, HPZ13}.

We prove $q$-duality relations for multiple $q$-harmonic non-strict sums
by employing  $q$-binomial duality relations found by Bradley \cite{Br05}.
We also establish a kind of duality for multiple $q$-harmonic strict sums by proving some new $q$-identities which are generalizations of the divisor generating functions identities due to Dilcher \cite{Di95} and Prodinger \cite{Pr00}.

\noindent Finally, we put forward the following conjecture concerning cyclic sums of some multiple $q$-harmonic sums.
\begin{conjecture}[Cyclic-sum]
Let $d_0,d_1,\dots,d_{t}$ be non-negative integers. Then

\begin{itemize}

\item[i)] For any prime $p>r+1$, where $r=\sum_{i=0}^{t}d_i+2t$,
\begin{equation}\label{CG12}
\sum_{i=0}^{t}H^q_{p-1}\left(
\{1\}^{d_i},2,\{1\}^{d_{i+1}},2,\dots,2,\{1\}^{d_{i+t}}
\right)\equiv
(-1)^t\binom{p+t}{r+1}\frac{(1-q)^{r}}{p}
\pmod{\qq{p}},
\end{equation}

\item[ii)] For any prime $p>r+1$, where $r=\sum_{i=0}^{t}2d_i+3t$,
\begin{equation}\label{CG23}
\sum_{i=0}^{t}H^q_{p-1}\left(
\{2\}^{d_i},3,\{2\}^{d_{i+1}},3,\dots,3,\{2\}^{d_{i+t}}
\right)\equiv
N(p,r,t)\frac{(1-q)^{r}}{p}
\pmod{\qq{p}},
\end{equation}
where $N(p,r,t)$ does not depend on $q$.
\end{itemize}
In both congruences it is understood that $d_i=d_j$ if $i\equiv j$ modulo $t+1$.
\end{conjecture}
Note that the case $t=1$ is exactly Theorem \ref{T6.2} and Theorem \ref{Tx.x} below. The case of arbitrary $t$ when all $d_j$ are zero follows from Theorem \ref{T3.1} and Theorem \ref{T33.3}.

\section{Preliminary results about $K_n(p)$}

We start with recalling several known  results on the sequence $K_n(p)$ defined by (\ref{qcon-1}).
The first few values of $K_n(p)$ are as follows:
\begin{equation*}
\begin{split}
K_0(p)&=-1, \,\, K_1(p)=-\frac{p-1}{2}, \,\, K_2(p)=-\frac{p^2-1}{12}, \,\, K_3(p)=-\frac{p^2-1}{24}, \\[3pt]
K_4(p)&=\frac{(p^2-1)(p^2-19)}{720}, \,\,\,\,\, K_5(p)=\frac{(p^2-1)(p^2-9)}{480}, \,\,\ldots.
\end{split}
\end{equation*}
The properties of this sequence were studied in detail by Dilcher in \cite{Di08} (just with another notation $\widetilde{D}_n(p)=-p^nK_n(p)$).
In particular, it was shown that for $n\ge 2,$ $K_n(p)$ is a polynomial in $p^2$ of degree at most $[n/2]$ which is divisible by $p^2-1.$
%From the partial fraction decomposition
%$$
%\frac{p}{z^p-1}=\frac{1}{z-1}+\sum_{j=1}^{p-1}\frac{\xi^j}{z-\xi^j}, \qquad \text{where} \,\,\, \xi=e^{2\pi i/p},
%$$
%it follows that
%\begin{equation} \label{qcon-0}
%K_n(p)=\sum_{j=1}^{p-1}\frac{\xi^j}{(1-\xi^j)^n}.
%\end{equation}
The sequence $K_n(p)$ is very closely related to the degenerate Bernoulli numbers (in fact, polynomials) that were first studied by Carlitz \cite{Ca56},
and can be defined from the expansion
\begin{equation} \label{qcon-2}
\frac{x}{(1+\lambda x)^{1/\lambda}-1}=\sum_{n=0}^{\infty}\beta_n(\lambda)\frac{x^n}{n!}.
\end{equation}
By comparing generating functions (\ref{qcon-1}) and (\ref{qcon-2}) we have the relation (see \cite[Theorem 3]{Di08})
\begin{equation} \label{knp}
K_n(p)=\frac{(-1)^{n-1}}{n!} p^n \beta_n\left(\frac{1}{p}\right).
\end{equation}
Howard \cite{Ho} found explicit formulas for the coefficients of $\beta_n(\lambda)$ which together with (\ref{knp}) imply that for $n\ge 2,$
$$
K_n(p)=\frac{(-1)^{n-1}}{(n-1)!}\left(\frac{b_n}{n}
+\sum_{j=1}^{\lfloor n/2\rfloor}\frac{B_{2j}}{2j}\,s(n-1,2j-1)p^{2j}\right),
$$
where $b_n$ is the Bernoulli number of the second kind, $B_{2j}$ is the ordinary Bernoulli number, and $s(n,j)$ is the Stirling number of the first kind,
which are defined by the corresponding generating functions:
$$\frac{x}{\ln(1+x)}=\sum_{n=0}^{\infty}\frac{b_nx^n}{n!},\qquad
\frac{x}{e^x-1}=\sum_{n=0}^{\infty}\frac{B_nx^n}{n!},\qquad
\prod_{j=0}^{n-1}(x-j)=\sum_{j=0}^{n}s(n,j)x^j.
$$
From the other side, the ordinary Bernoulli numbers $B_n$
are the limit case of the degenerate Bernoulli numbers:
$$
\lim_{\lambda\to 0} \beta_n(\lambda)=B_n.
$$

\section{Preliminary results about $q$-MHS}

In this section, we describe some basic relations between multiple $q$-harmonic sums modulo $[p]_q$, which will be useful in the sequel.
These relations arise from inverting the order of summation of a nested sum and from expanding the product of two multiple $q$-harmonic sums.

\noindent Let $\overline{\bf s}=(s_l, s_{l-1}, \ldots, s_1)$ denote the reversal of ${\bf s}=(s_1, \ldots, s_{l-1}, s_l).$
Then we have the following relations.
\begin{theorem} \label{T3}
Let $p$ be a prime and ${\bf s}=(s_1,\ldots, s_l),$ ${\bf t}=(t_1,\ldots, t_l)$ be two $l$-tuples of non-negative integers. Then
\begin{align*}
H_{p-1}^q({\bf s}; {\bf t})&\equiv (-1)^{w({\bf s})}H_{p-1}^q(\overline{\bf s}; \overline{{\bf s}}-\overline{{\bf t}}) \pmod{[p]_q},\\
S_{p-1}^q({\bf s}; {\bf t})&\equiv (-1)^{w({\bf s})}S_{p-1}^q(\overline{\bf s}; \overline{{\bf s}}-\overline{{\bf t}}) \pmod{[p]_q}.
\end{align*}
In particular,
\begin{align}
H_{p-1}^q({\bf s})&\equiv (-1)^{w({\bf s})}H_{p-1}^q(\overline{\bf s}; \{1\}^l) \pmod{[p]_q},\label{EE1}\\
S_{p-1}^q({\bf s})&\equiv (-1)^{w({\bf s})}S_{p-1}^q(\overline{\bf s}; \{1\}^l) \pmod{[p]_q}.\label{EE2}
\end{align}
\end{theorem}
\begin{proof}
Reversing the order of summation, we get
\begin{equation*}
\begin{split}
H_{p-1}^q({\bf s}; {\bf t})&=\sum_{1\le k_1<\cdots<k_l<p}\frac{q^{t_1k_1+\cdots+t_lk_l}}{[k_1]_q^{s_1}\cdots [k_l]_q^{s_l}}
=\sum_{p>k_1>\cdots>k_l\ge 1}\frac{q^{t_1(p-k_1)+\cdots+t_l(p-k_l)}}{[p-k_1]_q^{s_1}\cdots[p-k_l]_q^{s_l}}\\
&=(1-q)^{w({\bf s})}\sum_{1\le k_l<\cdots<k_1<p}\frac{q^{t_1(p-k_1)+\cdots+t_l(p-k_l)+k_1s_1+\cdots+k_ls_l}}{(q^{k_1}-q^p)^{s_1}\cdots (q^{k_l}-q^p)^{s_l}} \\
&\equiv (-1)^{w({\bf s})}(1-q)^{w({\bf s})}\sum_{1\le k_l<\cdots<k_1<p}\frac{q^{k_1(s_1-t_1)+\cdots+k_l(s_l-t_l)}}{(1-q^{k_1})^{s_1}\cdots (1-q^{k_l})^{s_l}}\\
&=(-1)^{w({\bf s})}H_{p-1}^q(\overline{\bf s}; \overline{{\bf s}}-\overline{{\bf t}}) \pmod{[p]_q}.
\end{split}
\end{equation*}
Putting ${\bf t}=(s_1-1,\ldots,s_l-1),$ we get (\ref{EE1}).
The proofs for the non-strict sums are similar.
\end{proof}

For depth one $q$-harmonic sums $H_{p-1}^q(s),$ we have the following.
\begin{corollary} \label{T2}
If $p\ge 3$ is a prime, then for all positive integers  $s$, we have
\begin{equation*}
H_{p-1}^q(s)\equiv (-1)^s K_s(p) (1-q)^s \pmod{[p]_q}.
\end{equation*}
\end{corollary}
\begin{proof}
Setting $l=1$ in (\ref{EE1}) and applying (\ref{qcon-4}), we get the desired congruence.
\end{proof}

It is easy to show that the product of two multiple harmonic sums $H_n^q$ (resp.~$S_n^q$) can be expressed as a linear combination of $H_n^q$
(resp.~$S_n^q$). For example,
\begin{align*}
H_n^q(s_1)H_n^q(s_2)&=H_n^q(s_1,s_2)+H_n^q(s_2,s_1)+ H_n^q(s_1+s_2)+(1-q)H_n^q(s_1+s_2-1), \\[3pt]
S_n^q(s_1)S_n^q(s_2)&=S_n^q(s_1,s_2)+S_n^q(s_2,s_1)-S_n^q(s_1+s_2)-(1-q)S_n^q(s_1+s_2-1).
\end{align*}
From the above relations and Corollary \ref{T2} we obtain the following congruences.
\begin{lemma} \label{L1}
Let $p$ be a prime and $a, b$ be positive integers. Then we have modulo $[p]_q$,
\begin{align*}
H_{p-1}^q(a,b)+H_{p-1}^q(b,a)&\equiv (-1)^{a+b}(1-q)^{a+b}(K_a(p)K_b(p)-K_{a+b}(p)+K_{a+b-1}(p)),  \\[3pt]
S_{p-1}^q(a,b)+S_{p-1}^q(b,a)&\equiv (-1)^{a+b}(1-q)^{a+b}(K_a(p)K_b(p)+K_{a+b}(p)-K_{a+b-1}(p)).
\end{align*}
\end{lemma}

\section{Homogeneous sums $H_{p-1}^q(\{2\}^t)$ and $S_{p-1}^q(\{2\}^t)$}

In this section, we prove congruences for homogeneous strict and non-strict sums on strings composed of twos.
\begin{theorem} \label{T3.1}
For any integer $t\ge 1$ and any prime $p>t$, we have
\begin{equation}\label{CC2}
H_{p-1}^q(\{2\}^t)\equiv (-1)^t\binom{p+t}{2t+1} \frac{(1-q)^{2t}}{p(t+1)} \pmod{[p]_q}.
\end{equation}
\end{theorem}
\begin{proof}
We start by considering a generating function for the sequence $H_{p-1}^q(\{1\}^l),$ which is given by the product
\begin{equation*}
\prod_{k=1}^{p-1}\left(1+\frac{x}{1-q^k}\right)=\sum_{l=0}^{p-1}\frac{H_{p-1}^q(\{1\}^l)}{(1-q)^l}\,x^l.
\end{equation*}
By \cite[Theorem 2.1]{Zh13}, we have
\begin{equation} \label{eq:gen1}
\prod_{k=1}^{p-1}\left(1+\frac{x}{1-q^k}\right)\equiv -\frac{1}{px}\left(1-(1+x)^p\right) \pmod{[p]_q}.
\end{equation}
We are going to relate our sequence $H_{p-1}^q(\{2\}^t)$ to the finite product
\begin{equation} \label{eq:gen2}
F(z):=\prod_{k=1}^{p-1}\left(1-\frac{q^k}{(1-q^k)^2}\frac{(z-1)^2}{z}\right),
\end{equation}
the limit case of which was studied in \cite{DHP14}.
By (\ref{eq:gen2}) and (\ref{eq:gen1}), we get
\begin{equation} \label{eq:gen3}
\begin{split}
F(z)&=\prod_{k=1}^{p-1}\left(1+\frac{z-1}{1-q^k}\right)\prod_{k=1}^{p-1}\left(1+\frac{z^{-1}-1}{1-q^k}\right) \\[3pt]
&\equiv
\frac{(1-z^p)(1-z^{-p})}{p^2(z-1)(z^{-1}-1)}=\frac{(z^p-1)^2}{p^2z^{p-1}(z-1)^2} \pmod{[p]_q}.
\end{split}
\end{equation}
From the other side, for any integer $t,$ $0\le t\le p-1,$ by expanding the product in (\ref{eq:gen2}), we have
\begin{equation*}
z^tF(z)=\sum_{l=0}^{p-1}\frac{(-1)^lH_{p-1}^q(\{2\}^l)}{(1-q)^{2l}}\,(z-1)^{2l}z^{t-l}.
\end{equation*}
By expanding $z^{t-l}$ in powers of $z-1,$
\begin{equation*}
z^{t-l}=(1+(z-1))^{t-l}=\sum_{k=0}^{\infty}\binom{t-l}{k}(z-1)^k
\end{equation*}
where
\begin{equation*}
\binom{\alpha}{k}=\frac{\alpha(\alpha-1)\cdots(\alpha-k+1)}{k!},
\end{equation*}
we obtain
\begin{equation} \label{eq:gen4}
\begin{split}
z^tF(z)&=\sum_{l=0}^{p-1}\sum_{k=0}^{\infty}(z-1)^{2l+k}\binom{t-l}{k}\frac{(-1)^lH_{p-1}^q(\{2\}^l)}{(1-q)^{2l}} \\[3pt]
&=\sum_{j=0}^{\infty}(z-1)^j\sum_{l=0}^{\min([j/2],p-1)}\binom{t-l}{j-2l}\frac{(-1)^lH_{p-1}^q(\{2\}^l)}{(1-q)^{2l}}.
\end{split}
\end{equation}
To extract $H_{p-1}^q(\{2\}^t)$ from (\ref{eq:gen4}), we consider the coefficient of $(z-1)^{2t}$ on the right-hand side, which is exactly
\begin{equation} \label{eq:gen5}
\sum_{l=0}^t\binom{t-l}{2t-2l}\frac{(-1)^lH_{p-1}^q(\{2\}^l)}{(1-q)^{2l}}=\frac{(-1)^tH_{p-1}^q(\{2\}^t)}{(1-q)^{2t}},
\end{equation}
since $\binom{t-l}{2t-2l}$ is distinct from zero for $l=0,1,\ldots, t$ if and only if $l=t.$

Summarizing (\ref{eq:gen3}), (\ref{eq:gen4}) and (\ref{eq:gen5}), we get
\begin{equation} \label{eq:gen6}
\frac{(-1)^tH_{p-1}^q(\{2\}^t)}{(1-q)^{2t}}\equiv [(z-1)^{2t}]\frac{(z^p-1)^2z^{t-p+1}}{p^2(z-1)^2} \pmod{[p]_q},
\end{equation}
where $[(z-a)^j]f(z)$ is the coefficient of $(z-a)^j$ in Taylor's expansion of $f(z)$ centered at~$a$.
By expanding $\frac{(z^p-1)^2 z^{t-p+1}}{p^2(z-1)^2}$ into powers of $z-1,$ we have
\begin{equation*}
\begin{split}
\frac{(z^p-1)^2 z^{t-p+1}}{p^2(z-1)^2}&=\frac{1}{p^2(z-1)^2}(z^{p+t+1}-2z^{t+1}+z^{t-p+1}) \\
&=\frac{1}{p^2(z-1)^2}\left((1+(z-1))^{p+t+1}-2(1+(z-1))^{t+1}+(1+(z-1))^{t-p+1}\right) \\
&=\frac{1}{p^2(z-1)^2}\left(\sum_{k=0}^{p+t+1}\binom{p+t+1}{k}(z-1)^k-2\sum_{k=0}^{t+1}\binom{t+1}{k}(z-1)^k \right.\\
&+
\left.\sum_{k=0}^{\infty}\binom{t-p+1}{k}(z-1)^k\right),
\end{split}
\end{equation*}
and therefore by (\ref{eq:gen6}),
\begin{equation*}
\frac{(-1)^tH_{p-1}^q(\{2\}^t)}{(1-q)^{2t}}\equiv \frac{1}{p^2}\left(\binom{p+t+1}{2t+2}+\binom{t-p+1}{2t+2}\right)
=\frac{1}{p(t+1)}\binom{p+t}{2t+1} \pmod{[p]_q},
\end{equation*}
that completes the proof.
\end{proof}
By applying a similar argument as above, we are able to find the right-hand side of congruence (\ref{222}) explicitly.
\begin{theorem}
For any positive integer $t$ and any prime $p>t,$ we have
\begin{equation}\label{qcon-13.5}
H_{p-1}^q(\{2\}^t;\{0\}^t)\equiv \left((-1)^t\binom{p-1}{2t+1}+\binom{p-1}{t}\right) \frac{(1-q)^{2t}}{(t+1)p} \pmod{[p]_q}.
\end{equation}
\end{theorem}
\begin{proof}
 Consider the product
\begin{equation} \label{qcon-14}
F_1(z)=\prod_{k=1}^{p-1}\left(1-\frac{z^2}{(1-q^k)^2}\right)=\sum_{t=0}^{p-1}\frac{(-1)^tH_{p-1}^q(\{2\}^t;\{0\}^t)}{(1-q)^{2t}}z^{2t}.
\end{equation}
Similarly as in the previous proof, we have
\begin{equation} \label{qcon-15}
\begin{split}
F_1(z)&=\prod_{k=1}^{p-1}\left(1-\frac{z}{1-q^k}\right)\left(1+\frac{z}{1-q^k}\right)
\equiv \frac{-1}{p^2z^2}(1-(1-z)^p)(1-(1+z)^p) \\
&=\frac{-1}{p^2z^2}(1-(1+z)^p-(1-z)^p+(1-z^2)^p)
 \pmod{[p]_q}.
\end{split}
\end{equation}
Now by comparing coefficients of $z^{2t}$ on the right-hand sides of (\ref{qcon-14}) and (\ref{qcon-15}), we easily get
$$
\frac{(-1)^tH_{p-1}^q(\{2\}^t;\{0\}^t)}{(1-q)^{2t}}\equiv \frac{-1}{p^2}\left(-\binom{p}{2t+2}-\binom{p}{2t+2}+(-1)^{t+1}\binom{p}{t+1}\right)
$$
which implies (\ref{qcon-13.5}).
\end{proof}
\begin{theorem} \label{T3.3}
For any  integer $t\ge 0$  and
for any prime $p>2t+1$,
\begin{equation*}
S^q_{p-1}(\{2\}^t)\equiv
(1-q)^{2t}\sum_{i=0}^t
\left(t\binom{t-1}{i}+(t-1)\binom{t}{i}\right)
(-1)^{i}K_{2t-i}(p)
\pmod{\qq{p}}.
\end{equation*}
\end{theorem}
\begin{proof}
By (23) in \cite{HP13}, we have
$$S^q_n(\{ 2\}^t)=\sum_{k=1}^n
\qbin{n}{k}\qbin{n+k}{k}^{-1}\frac{(-1)^{k-1}(1+q^k)q^{\binom{k}{2}+tk}}{\qq{k}^{2t}},
$$
where
$$\qbin{n}{k}=\prod_{j=1}^k\frac{1-q^{n-k+j}}{1-q^j}$$
is the {\sl Gaussian $q$-binomial coefficient}.
By (17) and (18) in \cite{Ta13}, for $k=1,\dots,p-1$,
$$\qq{p}\qbin{p-1}{k}\qbin{p-1+k}{k}^{-1}\equiv
(-1)^{k}q^{-\binom{k}{2}-k}
\left(\qq{k}-\qq{p}
-\qq{p}\qq{k}\sum_{j=1}^{k-1}\frac{1+q^j}{\qq{j}}
\right)\!\!\pmod{\qq{p}^2}.$$
Hence, for $n=p-1$ with $p>2t+1$, we obtain
\begin{align*}
S^q_{p-1}(\{ 2\}^t)
&\equiv\frac{1}{\qq{p}}\sum_{k=1}^{p-1}
\frac{(1+q^k)q^{(t-1)k}}{\qq{k}^{2t}}
\left(-\qq{k}+\qq{p}
+\qq{p}\qq{k}\sum_{j=1}^{k-1}\frac{1+q^j}{\qq{j}}\right)\\
&\equiv-\frac{1}{\qq{p}}\sum_{k=1}^{p-1}
\frac{q^{(t-1)k}+q^{tk}}{\qq{k}^{2t-1}}
 +\sum_{k=1}^{p-1}
\frac{q^{(t-1)k}+q^{tk}}{\qq{k}^{2t}}\\
&\qquad +\sum_{1\leq j<k\leq p-1}
\frac{(q^{tk}+q^{j+(t-1)k})+(q^{(t-1)k}+q^{j+tk})}{\qq{j}\qq{k}^{2t-1}}
\pmod{\qq{p}}.
\end{align*}
By (13) in \cite{Ta13}, we have
$$\sum_{1\leq j<k\leq p-1}
\frac{q^{tk}+q^{j+(t-1)k}}{\qq{j}\qq{k}^{2t-1}}
\equiv
\sum_{k=1}^{p-1}\frac{1}{\qq{j}}
\sum_{k=1}^{p-1}\frac{q^{tk}}{\qq{k}^{2t-1}}
-\sum_{k=1}^{p-1}\frac{q^{tk}}{\qq{k}^{2t}}
\pmod{\qq{p}},$$
and
$$\sum_{1\leq j<k\leq p-1}
\frac{q^{(t-1)k}+q^{j+tk}}{\qq{j}\qq{k}^{2t-1}}
\equiv
\sum_{k=1}^{p-1}\frac{1}{\qq{j}}
\sum_{k=1}^{p-1}\frac{q^{(t-1)k}}{\qq{k}^{2t-1}}
-\sum_{k=1}^{p-1}\frac{q^{(t-1)k}}{\qq{k}^{2t}}
\pmod{\qq{p}}.$$
Hence
\begin{align*}
S^q_{p-1}(\{ 2\}^t)
&\equiv
\left(-\frac{1}{\qq{p}}+\sum_{k=1}^{p-1}\frac{1}{\qq{j}}\right)
\sum_{k=1}^{p-1}
\frac{q^{(t-1)k}+q^{tk}}{\qq{k}^{2t-1}}\\
&\equiv
(2t-1)\sum_{k=1}^{p-1}\frac{q^{tk}}{\qq{k}^{2t}}
-(t-1)(1-q)\sum_{k=1}^{p-1}\frac{q^{tk}}{\qq{k}^{2t-1}}
\pmod{\qq{p}},
\end{align*}
where we used (11) in \cite{Ta13}.
Finally, we apply \eqref{HS}, which completes the proof.
\end{proof}

\section{Homogeneous sums $H_{p-1}^q(\{3\}^t)$}

In this section, we consider homogeneous sums $H_{p-1}^q(\{s\}^t)$ when $s=3$. Before proving the main result of this section, we state the following
combinatorial lemma.
We owe the proof of this lemma to Robin Chapman \cite{RC15}.
\begin{lemma} \label{CL}
Let $p, s$ be positive integers, $\delta\in\{0, 1\}$ and
\begin{align*}
T(p,s,\delta)&:=[x^{3s}]\sum_{k=0}^{p-\delta}(-1)^k\binom{p-k-\delta}{k-\delta}(2-x)^{p-2k}(1+x)^{s-k}, \\
\widetilde{T}(p,s,\delta)&:=[x^{3s}]\sum_{k=0}^{p-\delta}(-1)^k\binom{p-k-\delta}{k-\delta}(2-x)^{p-2k}(1+x)^{p+s-k}.
\end{align*}
Then
$$
T(p,s,\delta)=(-1)^{\delta}\binom{s+p-2\delta+1}{3s+1}, \quad
\widetilde{T}(p,s,\delta)=(-1)^{s+\delta}\binom{2s+p-2\delta+1}{3s+1}.
$$
\end{lemma}
\begin{proof}
Consider  generating functions of both sequences
$$
G_{s,\delta}(y)=(-1)^{\delta}\sum_{p=1}^{\infty}T(p,s,\delta)y^p \quad\text{and}\quad
\widetilde{G}_{s,\delta}(y)=(-1)^{\delta}\sum_{p=1}^{\infty}\widetilde{T}(p,s,\delta)y^p.
$$
Then we have
\begin{align*}
G_{s,\delta}(y)
&=
[x^{3s}]\sum_{p=1}^{\infty}\sum_{k=0}^{p-2\delta}(-1)^{k}\binom{p-k-2\delta}{k}
(2-x)^{p-2k-2\delta}(1+x)^{s-k-\delta}y^p\\
&=
y^{2\delta}[x^{3s}](1+x)^{s-\delta}\sum_{r\ge 0}\sum_{k=0}^r(-1)^k\binom{r}{k}
(2-x)^{r-k}(1+x)^{-k}y^{k+r}\\
&=y^{2\delta}[x^{3s}](1+x)^{s-\delta}
\sum_{r\ge 0}y^r(2-x)^r\left(1-\frac{y}{(2-x)(1+x)}\right)^{r}\\
&=y^{2\delta}[x^{3s}]\frac{(1+x)^{s-\delta}}{
1-y(2-x)\left(1-\frac{y}{(2-x)(1+x)}\right)}\\
&=y^{2\delta}[x^{3s}]\frac{(1+x)^{s-\delta+1}}{
1+x-y(2-x)(1+x)+y^2}\\
&=y^{2\delta}[x^{3s}]\frac{(1+x)^{s-\delta+1}}{
(1-y)^2+(1-y)x+yx^2}.
\end{align*}
Hence, by letting $u=1/(1-y)$, we easily obtain
\begin{align*}
(-1)^{\delta}T(p,s,\delta)
&=[y^{p-2\delta}][x^{3s}]
\frac{(1+x)^{s-\delta+1}}{
(1-y)^2+(1-y)x+yx^2}\\
&=[y^{p-2\delta}]\frac{[x^{3s}]}{(1-y)^2}
\frac{(1+x)^{s-\delta+1}}{
1+ux+(u^2-u)x^2}\\
&=[y^{p-2\delta}]\frac{[x^{3s}]}{(1-y)^2}
(1+x)^{s-\delta+1}\sum_{j=0}^{\infty}a_j x^j\\
&=[y^{p-2\delta}]\frac{1}{(1-y)^2}
\sum_{j=0}^{s-\delta+1}\binom{s-\delta+1}{j}
a_{2s+\delta-1+j} \\
&=[y^{p-2\delta}]\frac{1}{(1-y)^2}
\sum_{j=0}^{s-\delta+1}\binom{s-\delta+1}{j}
\frac{\alpha^{2s+\delta+j}-\beta^{2s+\delta+j}}{\alpha-\beta} \\
&=[y^{p-2\delta}]\frac{1}{(1-y)^2}
\frac{\alpha^{2s+\delta}(1+\alpha)^{s-\delta+1}-\beta^{2s+\delta}(1+\beta)^{s-\delta+1}}{\alpha-\beta} \\
&=[y^{p-2\delta}]\frac{y^{2s}}{(1-y)^{3s+2}}
\frac{\alpha^{\delta}(1+\alpha)^{1-\delta}-\beta^{\delta}(1+\beta)^{1-\delta}}{\alpha-\beta} \\
&=[y^{p-2\delta-2s}](1-y)^{-(3s+2)}
=\binom{s+p-2\delta+1}{3s+1}
\end{align*}
where we used the fact that $\alpha^2(1+\alpha)=\beta^2(1+\beta)=u(u-1)^2=y^2/(1-y)^3$.

Similarly, for the second generating function, we have
\begin{align*}
\widetilde{G}_{s,\delta}(y)
&=
[x^{3s}]\sum_{p=1}^{\infty}\sum_{k=0}^{p-2\delta}(-1)^{k}\binom{p-k-2\delta}{k}
(2-x)^{p-2k-2\delta}(1+x)^{p+s-k-\delta}y^p\\
&=
y^{2\delta}[x^{3s}](1+x)^{s+\delta}\sum_{r\ge 0}\sum_{k=0}^r(-1)^k\binom{r}{k}
(2-x)^{r-k}(1+x)^{r}y^{k+r}\\
&=y^{2\delta}[x^{3s}](1+x)^{s+\delta}
\sum_{r\ge 0}y^r(2-x)^r(1+x)^r\left(1-\frac{y}{2-x}\right)^{r}\\
&=y^{2\delta}[x^{3s}]\frac{(1+x)^{s+\delta}}{
1-y(1+x)(2-x-y)}\\
&=y^{2\delta}[x^{3s}]\frac{(1+x)^{s+\delta}}{
(1-y)^2-y(1-y)x+yx^2}.
\end{align*}
By comparing the coefficients of the powers of $y$ on both sides, we get
\begin{align*}
(-1)^{\delta}\widetilde{T}(p,s,\delta)
&=[y^{p-2\delta}][x^{3s}]
\frac{(1+x)^{s+\delta}}{
(1-y)^2-y(1-y)x+yx^2}\\
&=[y^{p-2\delta}]\frac{[x^{3s}]}{(1-y)^2}
\frac{(1+x)^{s+\delta}}{
1-(u-1)x+(u^2-u)x^2}\\
&=[y^{p-2\delta}]\frac{[x^{3s}]}{(1-y)^2}
(1+x)^{s+\delta}\sum_{j=0}^{\infty}\widetilde{a}_j x^j\\
&=[y^{p-2\delta}]\frac{1}{(1-y)^2}
\sum_{j=0}^{s+\delta}\binom{s+\delta}{j}
\widetilde{a}_{2s-\delta+j} \\
&=[y^{p-2\delta}]\frac{1}{(1-y)^2}
\sum_{j=0}^{s+\delta}\binom{s+\delta}{j}
\frac{\widetilde{\alpha}^{2s-\delta+j+1}-\widetilde{\beta}^{2s-\delta+j+1}}{\widetilde{\alpha}-\widetilde{\beta}} \\
&=[y^{p-2\delta}]\frac{1}{(1-y)^2}
\frac{\widetilde{\alpha}^{2s-\delta+1}(1+\widetilde{\alpha})^{s+\delta}-
\widetilde{\beta}^{2s-\delta+1}(1+\widetilde{\beta})^{s+\delta}}{\widetilde{\alpha}-\widetilde{\beta}} \\
&=[y^{p-2\delta}]\frac{(-1)^sy^{s}}{(1-y)^{3s+2}}
\frac{\widetilde{\alpha}^{1-\delta}(1+\widetilde{\alpha})^{\delta}-\widetilde{\beta}^{1-\delta}(1+\widetilde{\beta})^{\delta}}{\widetilde{\alpha}-\widetilde{\beta}} \\
&=(-1)^s[y^{p-2\delta-s}](1-y)^{-(3s+2)}
=(-1)^s\binom{2s+p-2\delta+1}{3s+1}
\end{align*}
where we used the fact that $\widetilde{\alpha}^2(1+\widetilde{\alpha})=\widetilde{\beta}^2(1+\widetilde{\beta})=u^2(1-u)=-y/(1-y)^3$.
\end{proof}

\begin{theorem} \label{T33.3}
For any positive integer $t$ and for any prime $p>t$,
\begin{equation}\label{33}
\frac{H_{p-1}^q(\{3\}^t)}{(1-q)^{3t}}\equiv
\frac{1}{(t+1)p^2}
\left(\binom{p+2t+1}{3t+2}+\binom{-p+2t+1}{3t+2}\right)
 \pmod{[p]_q}.
\end{equation}
\end{theorem}

\begin{proof}
Since by (\ref{EE1}), $H_{p-1}^q(\{3\}^t)\equiv (-1)^t H_{p-1}^q(\{3\}^t; \{1\}^t)
\pmod{[p]_q}$, we may deal with $H_{p-1}^q(\{3\}^t; \{1\}^t)$.
A generating function for the sequence $H_{p-1}^q(\{3\}^t; \{1\}^t)$
can be derived from the product
$$
F(z)=\prod_{k=1}^{p-1}\left(1+\frac{q^k}{(1-q^k)^3}\frac{(z-1)^3}{z}\right).
$$
We have
$$
z^tF(z)=\sum_{l=0}^{p-1}\frac{H_{p-1}^q(\{3\}^l; \{1\}^l)}{(1-q)^{3l}}\,(z-1)^{3l}z^{t-l}
$$
and since
$$
z^{t-l}=(1+(z-1))^{t-l}=\sum_{k=0}^{\infty}\binom{t-l}{k}(z-1)^k,
$$
we get
\begin{equation*}
\begin{split}
z^tF(z)&=\sum_{l=0}^{p-1}\sum_{k=0}^{\infty}
(z-1)^{3l+k}\binom{t-l}{k}\frac{H_{p-1}^q(\{3\}^l; \{1\}^l)}{(1-q)^{3l}} \\
&=\sum_{j=0}^{\infty}(z-1)^j
\sum_{l=0}^{\min([j/3],p-1)}\binom{t-l}{j-3l}\frac{H_{p-1}^q(\{3\}^l; \{1\}^l)}{(1-q)^{3l}}.
\end{split}
\end{equation*}
To extract $H_{p-1}^q(\{3\}^t; \{1\}^t),$ we consider
\begin{equation} \label{11}
[(z-1)^{3t}]z^tF(z)=\sum_{l=0}^{t}\binom{t-l}{3t-3l}
\frac{H_{p-1}^q(\{3\}^l; \{1\}^l)}{(1-q)^{3l}}=
\frac{H_{p-1}^q(\{3\}^t; \{1\}^t)}{(1-q)^{3t}}.
\end{equation}
From the other side, we can factor $F(z)$ as
$$
F(z)=\prod_{k=1}^{p-1}\left(1+\frac{z-1}{1-q^k}\right)
\left(1-\frac{\frac{(z-1)}{2}(1+\sqrt{1-4/z})}{1-q^k}
\right)\left(1-\frac{\frac{(z-1)}{2}(1-\sqrt{1-4/z})}{1-q^k}
\right).
$$
Then by (\ref{eq:gen1}) we have
\begin{align*}
F(z)\equiv -\frac{(1-z^p)}{p(z-1)}\cdot&
\frac{\bigl(1-(1-\frac{(z-1)}{2}(1+\sqrt{1-4/z}))^p\bigr)}
{p(z-1)(1+\sqrt{1-4/z})}\\
&\qquad\quad\times \frac{\bigl(1-(1-\frac{(z-1)}{2}(1-\sqrt{1-4/z}))^p\bigr)}
{p(z-1)(1-\sqrt{1-4/z})} \pmod{[p]_q},
\end{align*}
or
$$F(z)\equiv
\frac{z(z^p-1)}{p^3(z-1)^3}\left(1+\frac{1}{z^p}-
\frac{2}{z^{p/2}}T_p\left(\frac{\sqrt{z}(3-z)}{2}\right)\right)
 \pmod{[p]_q}.
$$
where $T_n(x)$ denotes the
Chebyshev polynomial of the first kind.

Hence, by (\ref{11}), we conclude that
\begin{equation} \label{oo}
\frac{H_{p-1}^q(\{3\}^t; \{1\}^t)}{(1-q)^{3t}}
\equiv
[(z-1)^{3t}]\frac{z^{t+1}(z^p-1)}{p^3(z-1)^3}
\left(1+\frac{1}{z^p}-
\frac{2}{z^{p/2}}T_p\left(\frac{\sqrt{z}(3-z)}{2}\right)\right) \pmod{[p]_q}.
\end{equation}
By Kummer's  formula
$$
a^n+b^n=\sum_{k=0}^{[n/2]}(-1)^k\frac{n}{n-k}\binom{n-k}{k}(ab)^k
(a+b)^{n-2k},
$$
we have
\begin{equation} \label{oa}
\frac{2}{z^{p/2}}T_p\left(\frac{\sqrt{z}(3-z)}{2}\right)=\sum_{k=0}^{(p-1)/2}(-1)^k\frac{p}{p-k}\binom{p-k}{k}\frac{1}{z^k}
(3-z)^{p-2k}.
\end{equation}
Taking into account that
$$\frac{p}{p-k}\binom{p-k}{k}=\binom{p-k}{k}+\binom{p-k-1}{k-1}$$
and by substituting (\ref{oa}) into (\ref{oo}), we get modulo $[p]_q$,
\begin{equation*}
\begin{split}
\frac{H_{p-1}^q(\{3\}^t; \{1\}^t)}{(1-q)^{3t}}
&\equiv \frac{1}{p^3}
[(z-1)^{3t+3}](z^{t+p+1}-z^{t+1-p}) \\
&-\frac{1}{p^3}(\widetilde{T}(p,t+1,0)-T(p,t+1,0)+\widetilde{T}(p,t+1,1)-T(p,t+1,1)).
\end{split}
\end{equation*}
Since
$$
[(z-1)^{3t+3}](z^{t+p+1}-z^{t+1-p})=\binom{p+t+1}{3t+3}+(-1)^t\binom{p+2t+1}{3t+3},
$$
by Lemma \ref{CL}, we readily get the desired congruence.
\end{proof}

\section{Quasi-homogeneous sums $H_{p-1}^q(\{2\}^a,3,\{2\}^b)$ and $S_{p-1}^q(\{2\}^a,3,\{2\}^b)$}

The aim of this section is to prove first congruences for quasi-homogeneous $q$-harmonic sums.

\begin{theorem} \label{Tx.x}
For any integers $a,b\geq 0$  and
for any prime $p>2t+1$ with $t=a+b+1$,
\begin{equation*}
\begin{split}
&H_{p-1}^q(\{2\}^a,3,\{2\}^b)+H_{p-1}^q(\{2\}^b,3,\{2\}^a)
\equiv
(-1)^{t-1}\binom{p+t}{2t+1}
\frac{(1-q)^{2t+1}}{p(t+1)}
\pmod{\qq{p}},\\
&S_{p-1}^q(\{2\}^a,3,\{2\}^b)+S_{p-1}^q(\{2\}^b,3,\{2\}^a) \\
&\qquad\qquad\quad\equiv
-(1-q)^{2t+1}\sum_{i=0}^t
\left(t\binom{t-1}{i}+(t-1)\binom{t}{i}\right)
(-1)^{i}K_{2t-i}(p)
\pmod{\qq{p}}.
\end{split}
\end{equation*}
\end{theorem}
\begin{proof}
We have
\begin{equation*}
\begin{split}
&\quad H_{p-1}^q(\{2\}^a,3,\{2\}^b)+(1-q)H_{p-1}^q(\{2\}^t) \\[3pt]
&=\sum_{k_1<k_2<\ldots<k_a}
\frac{q^{k_1+\dots+k_a}}{[k_1]_q^2\dots[k_a]_q^2}
\sum_{k_a<n<k_{a+1}}\frac{q^{2n}+q^n(1-q^n)}{[n]_q^3}
\sum_{k_{a+1}<k_{a+2}<\dots<k_{a+b}}
\frac{q^{k_{a+1}+\dots+k_{a+b}}}{[k_{a+1}]^2\dots[k_{a+b}]^2} \\[3pt]
&=H_{p-1}^q(\{2\}^a,3,\{2\}^b; \{1\}^{a+b+1})
\equiv -H_{p-1}^q(\{2\}^b,3,\{2\}^a) \pmod{\qq{p}},
\end{split}
\end{equation*}
where in the last congruence we used \eqref{EE1}.
Hence
$$H_{p-1}^q(\{2\}^a,3,\{2\}^b)+H_{p-1}^q(\{2\}^b,3,\{2\}^a)
\equiv
-(1-q)H_{p-1}^q(\{2\}^t)
\pmod{\qq{p}}$$
which, by \eqref{CC2}, implies the first  congruence of the theorem.
In a similar way, we show that
$$S_{p-1}^q(\{2\}^a,3,\{2\}^b)+S_{p-1}^q(\{2\}^b,3,\{2\}^a)
\equiv
-(1-q)S_{p-1}^q(\{2\}^t)
\pmod{\qq{p}}$$
and then by Theorem \ref{T3.3}, we conclude the proof.
\end{proof}

\section{Duality for multiple $q$-harmonic (non-strict) sums}

\noindent In \cite{Di95}, Dilcher established the following identity: for any pair
of positive integers $n$, $s$,
\begin{equation}\label{D}
\sum_{k=1}^n \qbin{n}{k}\frac{(-1)^{k}
q^{\binom{k}{2}+sk}}{\qq{k}^{s} }
=-\sum_{1\leq j_1\leq j_2\leq\cdots\leq j_s\le n}
\prod_{i=1}^s\frac{q^{j_i}}{\qq{j_i}}.
\end{equation}
Identity \eqref{D} is a generalization of the case $s=1$ due to Van Hamme \cite{VH82} and by
taking the limit as $n$ goes to infinity one obtains a remarkable $q$-series
related to overpartitions and divisor generating functions. For example, if $s=1$ and $|q|<1$ then
$$
\sum_{k=1}^{\infty} \frac{(-1)^{k-1}
q^{\binom{k+1}{2}}}{(1-q)(1-q^2)\cdots(1-q^{k-1})(1-q^{k})^2 }
=\sum_{j=1}^{\infty}\frac{q^j}{1-q^{j}}=\sum_{j=1}^{\infty}d(j)q^j.
$$
where $d(j)$ is the number of divisors of $j$ (see the pioneering paper of Uchimura \cite{Uc81}).

\noindent With this motivation, several authors have recently investigated \eqref{D}
and they extended it along several directions: see for example
\cite{An13, Br05, FL03, FL05, GZ11, GZ13, IS12, MSS13, Pr01, Pr04, Xu14, Ze05}.
In~\cite{Pr00}, Prodinger shows the inversion of \eqref{D},
\begin{equation}\label{H}
\sum_{k=1}^n \qbin{n}{k}(-1)^{k}
q^{\binom{k}{2}-nk}\!\!\!\!\!\sum_{1\leq j_1\leq j_2\leq\dots\leq j_s=k}
\prod_{i=1}^s\frac{q^{j_i}}{\qq{j_i}}
=-\sum_{k=1}^n\frac{q^{(s-1)k}}{\qq{k}^{s}},
\end{equation}
which is the $q$-analog of a formula of Hernandez \cite{He99}.

In \cite{Br05}, Bradley extended identity (\ref{D}) to multiple $q$-binomial sums
$$
A_n(s_1, \ldots, s_l)=
\sum_{k=1}^n\qbin{n}{k}(-1)^{k-1}
q^{\binom{k}{2}+k}\sum_{1\le k_1\le k_2\le\cdots \le k_l=k}
\prod_{j=1}^l\frac{q^{(s_j-1)k_j}}{[k_j]_q^{s_j}},
$$
(note that we changed the order of summation here in comparison with Bradley's definition)
and proved a duality identity between  $q$-harmonic non-strict sums $S_{n}^q$ and $A_n$:
for positive integers $r, a_1, b_1, \ldots, a_r, b_r,$
\begin{equation} \label{BrD}
S_n^q\Bigl(a_1, \{1\}^{b_1-1}, \overset{r}{\underset{j=2}{{\bf Cat}}}\{a_j+1, \{1\}^{b_j-1}\};{\bf 1}\Bigr)=
A_n\Bigl(\overset{r-1}{\underset{j=1}{{\bf Cat}}}\{\{1\}^{a_j-1}, b_j+1\}, \{1\}^{a_r-1}, b_r\Bigr)
\end{equation}
where ${\bf Cat}_{j=1}^l\{s_j\}$ stands for the concatenated argument sequence $s_1, \ldots, s_l$
and ${\bf 1}$ denotes a vector whose all components are equal to $1$.
From (\ref{BrD}) we get a duality relation for multiple $q$-harmonic non-strict sums modulo $[p]_q$.
\begin{theorem} \label{T4.1}
Let $p$ be a prime. Then for positive integers $r, a_1, b_1, \ldots, a_r, b_r,$ we have
\begin{equation*}
\begin{split}
S_{p-1}^q&\Bigl(\overset{r-1}{\underset{j=1}{{\bf Cat}}}\{\{1\}^{a_j-1}, b_j+1\}, \{1\}^{a_r-1}, b_r\Bigr) \\
&\equiv (-1)^{a_1+b_1+\cdots+a_r+b_r}%
S_{p-1}^q\Bigl(\overset{r-1}{\underset{j=1}{{\bf Cat}}}\{\{1\}^{b_{r-j+1}-1}, a_{r-j+1}+1\}, \{1\}^{b_1-1}, a_1\Bigr) \pmod{[p]_q}.
\end{split}
\end{equation*}
\end{theorem}
\begin{proof} For $1\leq j <p$, we have $\qq{p-j}=q^{-j}(\qq{p}-\qq{j})$, which implies that
\begin{equation} \label{nk}
\qbin{p-1}{k}
=\prod_{j=1}^{k}\frac{\qq{p-j}}{\qq{j}}
\equiv\prod_{j=1}^{k}(-q^{-j})
=(-1)^{k}q^{-\binom{k+1}{2}}
\pmod{\qq{p}}.
\end{equation}
Now by Theorem \ref{T3} and (\ref{BrD}), we readily get the congruence.
\end{proof}
In particular, for $r=1$ we obtain the following result.
\begin{corollary} \label{C3}
Let $p$ be a prime. Then for positive integers $a, b$ we have
$$
S_{p-1}^q(\{1\}^{a-1},b)\equiv(-1)^{a+b}\,S_{p-1}^q(\{1\}^{b-1},a) \pmod{[p]_q}.
$$
\end{corollary}

\begin{theorem}\label{T4.2}
Let $a, b$ be non-negative integers and $t$ be a positive integer. Then
\begin{equation*}
\begin{split}
&S_{p-1}^q(\{1\}^t)\equiv -K_t(p)(1-q)^t \pmod{[p]_q},\\
&S_{p-1}^q(1,\{2\}^{t-1},1)\equiv (1-q)^{2t}\sum_{j=0}^t\Bigl(t\binom{t-1}{j}+(t-1)\binom{t}{j}\Bigr)(-1)^{j+1}K_{2t-j}(p) \pmod{[p]_q}, \\
&S_{p-1}^q(\{1\}^{a},2,\{1\}^{b})+S_{p-1}^q(\{1\}^{b},2,\{1\}^{a}) \\
&\qquad\qquad\equiv (1-q)^{a+b+2}(K_{a+b+1}(p)-K_{a+b+2}(p)-K_{a+1}(p)K_{b+1}(p)) \pmod{[p]_q}.
\end{split}
\end{equation*}
\end{theorem}
\begin{proof} Setting $b=1,$ $a=t$ in Corollary \ref{C3}, we get the first congruence by Theorem \ref{T2}.
Setting $r=t,$ $a_1=\ldots=a_t=1,$ $b_1=\ldots=b_{t-1}=1,$ $b_t=2$ in Theorem \ref{T4.1}, we get
$$
S_{p-1}^q(\{2\}^t)\equiv -S_{p-1}^q(1,\{2\}^{t-1},1) \pmod{[p]_q},
$$
and the result follows immediately from Theorem \ref{T3.3}.
If we put $r=2,$ $a_1=a+1,$ $b_1=b_2=1,$ $a_2=b,$ we get
\begin{equation} \label{S121}
S_{p-1}^q(\{1\}^a,2,\{1\}^b)\equiv (-1)^{a+b+1}S_{p-1}^q(b+1,a+1) \pmod{[p]_q}.
\end{equation}
Now by Lemma \ref{L1}, we easily derive the last congruence of the theorem.
\end{proof}

\section{Duality for multiple $q$-harmonic (strict) sums}

\noindent In this section, we provide  further extensions of \eqref{D} and \eqref{H} and then consider application of the new identities to $q$-congruences.
\begin{theorem} \label{TwoId}
Let $n,s_1,s_2,\dots, s_l$ be positive integers. Then
\begin{align}\label{DE}
&\sum_{k=1}^n \qbin{n}{k}(-1)^{k}
q^{\binom{k}{2}+k}\!\!\!\!\sum_{1\leq k_1<k_2<\cdots<k_l=k}
\prod_{i=1}^l\frac{q^{(s_i-1)k_i}}{\qq{k_i}^{s_i}}
=(-1)^l\!\!\!
\underset{j_i<j_{i+1},\; i\in I}{\sum_{1\leq j_1\leq j_2\leq\dots\leq j_w\leq n}}
\prod_{i=1}^w\frac{q^{j_i}}{\qq{j_i}},\\
\label{HE}
&\sum_{k=1}^n \qbin{n}{k}(-1)^{k}
q^{\binom{k}{2}-nk}\!\!\!
\underset{j_i<j_{i+1},\; i\in I}{\sum_{1\leq j_1\leq j_2\leq\dots\leq j_w=k}}
\prod_{i=1}^w\frac{q^{j_i}}{\qq{j_i}}
=(-1)^l\!\!\!
\sum_{1\leq k_1<k_2<\cdots<k_l\leq n}
\prod_{i=1}^l\frac{q^{(s_i-1)k_i}}{\qq{k_i}^{s_i}}
\end{align}
where $w=\sum_{i=1}^ls_i$ and
$I=\{s_1,s_1+s_2,\dots,s_1+s_2+\dots+s_{l-1}\}$.
\end{theorem}
%Note that Theorem 2 in \cite{Br05} is the special case $s_1=s_2=\cdots=s_l=1$ of \eqref{DE}.
\noindent The proof of Theorem \ref{TwoId} will be given in the next section.

\noindent  From \eqref{DE} with $n$ replaced by $p-1$, where $p$ is a prime, with the help of congruence \eqref{nk} we get a kind of duality for multiple $q$-harmonic $H_n^q$ sums.
\begin{theorem} \label{T6.1}
Let $p$ be a prime and $s_1, \ldots, s_l$ be positive integers. Then
$$
H_{p-1}^q(s_1,\ldots,s_l)\equiv (-1)^l \underset{j_i<j_{i+1},\; i\in I}{\sum_{1\leq j_1\leq j_2\leq\dots\leq j_w\leq p-1}}
\prod_{i=1}^w\frac{q^{j_i}}{\qq{j_i}} \pmod{[p]_q},
$$
where $w=\sum_{i=1}^ls_i$ and
$I=\{s_1,s_1+s_2,\dots,s_1+s_2+\dots+s_{l-1}\}$.
\end{theorem}
\noindent
As an application of the duality relation above, we prove the following congruence for quasi-homogeneous sums.
\begin{theorem} \label{T6.2}
For any integers $a,b\geq 0$  and
for any prime $p>t$, where $t=a+b+1$,
\begin{equation}\label{CC1}
H^q_{p-1}(\{1\}^{a},2,\{1\}^{b})
+H^q_{p-1}(\{1\}^{b},2,\{1\}^{a})\equiv
-\binom{p+1}{t+2}\frac{(1-q)^{t+1}}{p}
\pmod{\qq{p}}.
\end{equation}
\end{theorem}
\begin{proof}
By Theorem \ref{T6.1}, we easily  obtain the following
congruences modulo $\qq{p}$,
\begin{align*}
-H^q_{p-1}(\{1\}^{a},2,\{1\}^{b})
&\equiv(-1)^{a+b}
\!\!\!
\sum_{1\leq j_1<\dots<j_{a}<j_{a+1}\leq  j_{a+2}<j_{a+3}
<\dots<j_{t+1}\leq p-1}
\prod_{i=1}^{t+1}\frac{q^{j_i}}{\qq{j_i}}\\
&\equiv(-1)^{a+b}
\!\!\!
\sum_{1\leq j_{t+1}<\dots<j_{a+3}<j_{a+2}\leq  j_{a+1}<j_{a}
<\dots<j_{1}\leq p-1}
\prod_{i=1}^{t+1}\frac{q^{p-j_i}}{\qq{p-j_i}}\\
&\equiv
\!\!\!
\sum_{1\leq j_{t+1}<\dots<j_{a+3}<j_{a+2}\leq  j_{a+1}<j_{a}
<\dots<j_{1}\leq p-1}
\prod_{i=1}^{t+1}\frac{1}{\qq{j_i}}.
\end{align*}
The right-hand side of the last congruence can be decomposed as
$$
\sum_{1\leq j_{t+1}<\dots<j_{a+3}<j_{a+2}\leq  j_{a+1}<j_{a}
<\dots<j_{1}\leq p-1}
\!\!\!\!\!\!\!\!\!\!\!\!\!\!\!\!\!\!\!\!\!\!\!\!
\!\!\!\!\!\!\!\!\!\!\!\!\!\!\!
((1-q)\qq{j_{a+1}}+q^{j_{a+1}})\prod_{i=1}^{r}\frac{1}{\qq{j_i}}
\;\;\;\;\;\;\;\;+\!\!\!\!\!\!\!\!\!\!\!\!
\sum_{1\leq j_{t+1}<\dots<j_{1}\leq p-1}
\prod_{i=1}^{t+1}\frac{1}{\qq{j_i}}.
$$
Therefore
$$-H^q_{p-1}(\{1\}^{a},2,\{1\}^{b})\equiv
(1-q)H^q_{p-1}(\{1\}^{t})
+H^q_{p-1}(\{1\}^{b},2,\{1\}^{a})+H^q_{p-1}(\{1\}^{t+1}).
$$
Finally,  by employing congruence (\ref{111}), we obtain
\begin{align*}
H^q_{p-1}(\{1\}^{a},2,\{1\}^{b})
&+H^q_{p-1}(\{1\}^{b},2,\{1\}^{a})
\equiv
-(1-q)H^q_{p-1}(\{1\}^{t})-H^q_{p-1}(\{1\}^{t+1})\\
&\equiv
-(1-q)\binom{p-1}{t}\frac{(1-q)^{t}}{r}
-\binom{p-1}{t+1}\frac{(1-q)^{t+1}}{t+2}\\
&\equiv
-\binom{p+1}{t+2}\frac{(1-q)^{t+1}}{p}
\pmod{\qq{p}}.
\end{align*}
\end{proof}

\section{Proof of Theorem \ref{TwoId}}

\noindent In order to prove identities \eqref{DE} and \eqref{HE}, we need first some auxiliary statements.
\begin{lemma} Let $n,j$ be positive integers with $j\leq n$. Then
\begin{align}\label{C1}
\sum_{k=j}^n \qbin{k-1}{j-1}q^k&=\qbin{n}{j}q^j,\\
\label{C2}
\sum_{k=j}^n\qbin{n}{k}(-1)^{k}q^{\binom{k}{2}}
&=\qbin{n-1}{j-1}(-1)^{j}q^{\binom{j}{2}}.
\end{align}
\end{lemma}
\begin{proof} Since
$$
\qbin{k}{j}=\qbin{k-1}{j-1}q^{k-j}+\qbin{k-1}{j},
$$
it follows that
$$\sum_{k=j}^n\qbin{k-1}{j-1}q^k=
q^j\sum_{k=j}^n\qbin{k-1}{j-1}q^{k-j}
=q^j\sum_{k=j}^n\left(\qbin{k}{j}-\qbin{k-1}{j}\right)=
\qbin{n}{j}q^j.$$
As regards \eqref{C2}, since
$$
\qbin{n}{k}=\qbin{n-1}{k-1}+q^k\qbin{n-1}{k},
$$
we have
\begin{align*}
\sum_{k=j}^n \qbin{n}{k}(-1)^{k}q^{\binom{k}{2}}&=
\sum_{k=j}^n\left(\qbin{n-1}{k-1}+q^k\qbin{n-1}{k}\right)(-1)^{k}q^{\binom{k}{2}}\\
&=\sum_{k=j}^n\left(\qbin{n-1}{k-1}(-1)^{k}q^{\binom{k}{2}}-
\qbin{n-1}{k}(-1)^{k+1}q^{\binom{k+1}{2}}\right)\\
&=\qbin{n-1}{j-1}(-1)^{j}q^{\binom{j}{2}}.
\end{align*}
\end{proof}

\noindent The following lemma provides some simple properties of the $q$-binomial transform of
$\{a_n\}_{n\geq 1}$ given by
\begin{equation}\label{BT}
b_n=\sum_{k=1}^n \qbin{n}{k}(-1)^{k}q^{\binom{k}{2}}a_k,
\end{equation}
and its inverse
\begin{equation}\label{IBT}
a_n=\sum_{k=1}^n \qbin{n}{k}(-1)^{k}q^{\binom{n-k}{2}}b_k.
\end{equation}

\begin{lemma} Let $\{a_n\}_{n\geq 1}$ and $\{b_n\}_{n\geq 1}$ be two sequences which satisfy \eqref{BT}.
Then for any positive integer $r$,
\begin{equation}\label{TC}
\sum_{k=1}^n \qbin{n}{k}\frac{(-1)^{k}
q^{\binom{k}{2}+rk}}{\qq{k}^{r}}\cdot a_{k}
=\sum_{1\leq k_1\leq k_2\leq\cdots\leq k_r\le n}
\prod_{i=1}^r\frac{q^{k_i}}{\qq{k_i}}\cdot b_{k_1},
\end{equation}
and
\begin{equation}\label{TCD}
\sum_{k=1}^n \qbin{n}{k}(-1)^{k}q^{\binom{k}{2}-nk}\!\!\!\!
\sum_{1\leq k_1\leq k_2\leq\cdots\leq k_r=k}
\prod_{i=1}^r\frac{q^{k_i}}{\qq{k_i}}\cdot b_{k_1}
=\sum_{k=1}^n
\frac{q^{(r-1)k}}{\qq{k}^r}\cdot a_{k}.
\end{equation}
\end{lemma}
\begin{proof}
For $i\geq 0$, let
$$b_n(i)=\sum_{k=1}^n \qbin{n}{k}(-1)^{k}q^{\binom{k}{2}}a_k(i)\\
\quad\mbox{with}\quad a_k(i)=\frac{q^{ik}a_k}{\qq{k}^i}.$$
Hence, by \eqref{C1},
\begin{align*}
\sum_{k=1}^n \frac{q^k b_k(i)}{\qq{k}}&=
\sum_{k=1}^n \frac{q^k}{\qq{k}}
\sum_{j=1}^k \qbin{k}{j}(-1)^{j}q^{\binom{j}{2}}a_j(i)
=\sum_{j=1}^n (-1)^{j}q^{\binom{j}{2}}a_j(i)
\sum_{k=j}^n \qbin{k}{j} \frac{q^k}{\qq{k}}\\
&=
\sum_{j=1}^n
\frac{(-1)^{j}q^{\binom{j}{2}}a_j(i)}{\qq{j}}
\sum_{k=j}^n \qbin{k-1}{j-1}q^k
=
\sum_{j=1}^n \qbin{n}{j}\frac{(-1)^{j}q^{\binom{j}{2}+j}a_j(i)}{\qq{j}}\\
&=\sum_{j=1}^n \qbin{n}{j} (-1)^{j}q^{\binom{j}{2}}a_j(i+1)=b_n(i+1).
\end{align*}
Therefore
$$
b_n(r)=
\sum_{k_r=1}^n \frac{q^{k_r} b_{k_r}(r-1)}{\qq{k_r}}
=\sum_{k_r=1}^n \frac{q^{k_r}}{\qq{k_r}}
\sum_{k_{r-1}=1}^{k_r} \frac{q^{k_{r-1}} b_{k_{r-1}}(r-2)}{\qq{k_{r-1}}}
=\!\!\!\!\!\sum_{1\leq k_1\leq k_2\leq\cdots\leq k_r\le n}
\prod_{i=1}^r\frac{q^{k_i}}{\qq{k_i}}\cdot b_{k_1}
$$
and \eqref{TC} is proved.

\noindent Finally, by Lemma 1 in \cite{Pr00} (or by the use of the inverse relation \eqref{IBT}),
it follows that
$$\sum_{k=1}^n \qbin{n}{k}(-1)^{k}q^{\binom{k}{2}-nk}(b_k(r)-b_{k-1}(r))=\sum_{k=1}^n q^{-k} a_k(r),$$
which yields \eqref{TCD}.
\end{proof}

\subsection{Proof of identities \eqref{DE} and \eqref{HE}}
We begin with the following result.

\begin{lemma}
Let ${\bf s}\in ({\mathbb N}^{*})^l$ and
let $r$ be a positive integer. Then
\begin{align} \label{I2}
\sum_{k=1}^n \qbin{n}{k}(-1)^{k}q^{\binom{k}{2}}
H^q_{k-1}({\bf s},r)&=
\frac{q^n}{\qq{n}}\sum_{k=1}^n \qbin{n}{k}\frac{(-1)^{k}q^{\binom{k}{2}+(r-1)k}H^q_{k-1}({\bf s})}{\qq{k}^{r-1}}\nonumber\\
&\qquad\qquad\qquad
-\sum_{k=1}^n \qbin{n}{k}
\frac{(-1)^{k}q^{\binom{k}{2}+rk}H^q_{k-1}({\bf s})}{\qq{k}^{r}}.
\end{align}
\end{lemma}
\begin{proof} By using \eqref{C2}, and the recursive relation
\begin{align*}
H^q_{n}({\bf s},r)&=
\sum_{k=1}^{n}
\frac{q^{(r-1)k}H^q_{k-1}({\bf s})}{\qq{k}^r},
\end{align*}
we have
\begin{align*}
\sum_{k=1}^n \qbin{n}{k}(-1)^{k}q^{\binom{k}{2}}
H^q_{k-1}({\bf s},r)&=
\sum_{k=1}^n \qbin{n}{k}(-1)^{k}q^{\binom{k}{2}}
\sum_{j=1}^{k-1}
\frac{q^{(r-1)j}H^q_{j-1}({\bf s})}{\qq{j}^r}\nonumber\\
&=
\sum_{j=1}^n \frac{q^{(r-1)j}H^q_{j-1}({\bf s})}{\qq{j}^r}
\sum_{k=j+1}^n\qbin{n}{k}(-1)^{k}q^{\binom{k}{2}}\nonumber\\
&=
\sum_{j=1}^n \frac{q^{(r-1)j}H^q_{j-1}({\bf s})}{\qq{j}^r}
\left(\qbin{n-1}{j}(-1)^{j+1}q^{\binom{j}{2}+j}\right)\nonumber\\
&=\sum_{j=1}^n \frac{(-q^j)\qq{n-j}}{\qq{n}\qq{j}}\qbin{n}{j}
\frac{(-1)^{j}q^{\binom{j}{2}+(r-1)j}H^q_{j-1}({\bf s})}{\qq{j}^{r-1}},
\end{align*}
and the proof is complete as soon as we note that
$$
\frac{q^n}{\qq{n}}-\frac{q^j}{\qq{j}}
=\frac{q^n(1-q^j)-q^j(1-q^n)}{\qq{n}\qq{j}(1-q)}
=\frac{(-q^j)\qq{n-j}}{\qq{n}\qq{j}}.
$$
\end{proof}

\noindent In order to reduce the use of symbols indicating multiple nested sums,
we introduce
$$
T^q_{n}({\bf s}):=\sum_{}
\prod_{i=1}^{w({\bf s})} \frac{q^{j_i}}{\qq{j_i}}
$$
where the sum is intended to be taken over all integers
satisfying the conditions:
$$\mbox{
$1\leq j_i\leq n$, $j_i<j_{i+1}$ for $i\in\{s_1,s_1+s_2,\dots,s_1+s_2+\dots+s_{l-1}\}$ and
$j_i\leq j_{i+1}$ otherwise.}$$
Note that the following recursive relations hold
\begin{align}\label{TT}
T^q_{n}({\bf s},1)=\sum_{k=1}^n \frac{q^k T^q_{k-1}({\bf s})}{\qq{k}}
\quad\mbox{and}\quad
T^q_{n}({\bf s},r)=\sum_{k=1}^n \frac{q^kT^q_{k}({\bf s},r-1)}{\qq{k}}\quad\mbox{for $r>1$}.
\end{align}

\noindent In the next theorem, the identities \eqref{I4}
and \eqref{I5} are equivalent to \eqref{DE} and \eqref{HE} respectively.

\begin{theorem}
Let ${\bf s}\in ({\mathbb N}^{*})^l$ and
let $r$ be a positive integer. Then
\begin{align}\label{I3}
&\sum_{k=1}^n \qbin{n}{k}(-1)^{k}
q^{\binom{k}{2}}H^q_{k-1}({\bf s})=
(-1)^{l+1} T^q_{n-1}({\bf s}),\\
&\label{I4}
\sum_{k=1}^n \qbin{n}{k}\frac{(-1)^{k}
q^{\binom{k}{2}+rk}H^q_{k-1}({\bf s})}{\qq{k}^{r}}=
(-1)^{l+1} T^q_{n}({\bf s},r),\\
&\label{I5}
\sum_{k=1}^n \qbin{n}{k}(-1)^{k}q^{\binom{k}{2}-nk}\!\!\!\!
\sum_{1\leq k_1\leq k_2\leq\cdots\leq k_r=k}
\prod_{i=1}^r\frac{q^{k_i}}{\qq{k_i}}\cdot T^q_{k_1-1}({\bf s})
=(-1)^{l+1} H^q_{n}({\bf s},r).
\end{align}
\end{theorem}
\begin{proof}
We proceed by induction on the length $l$.

\noindent The base case $l=0$ of \eqref{I3}
is true because by \eqref{C2},
$$\sum_{k=1}^n \qbin{n}{k}(-1)^{k} q^{\binom{k}{2}}=-1.$$
The base case $l=0$ of \eqref{I4} follows from \eqref{TC} and the previous equation, and reads
$$\sum_{k=1}^n \qbin{n}{k}\frac{(-1)^{k}
q^{\binom{k}{2}+rk}}{\qq{k}^{r} }=-T^q_n(r).$$
In a similar way, for $l=0$, \eqref{TCD} implies \eqref{I5}.

\noindent Now assume that $l>0$ and let ${\bf s}=({\bf t},s)$. By \eqref{I2}, we have
\begin{align*}
\sum_{k=1}^n \qbin{n}{k}(-1)^{k}
q^{\binom{k}{2}}H^q_{k-1}({\bf t}, s)
&=\frac{q^n}{\qq{n}}\sum_{k=1}^n \qbin{n}{k}\frac{(-1)^{k}q^{\binom{k}{2}+(s-1)k}H^q_{k-1}({\bf t})}{\qq{k}^{s-1}}\nonumber\\
&\qquad\qquad\qquad
-\sum_{k=1}^n \qbin{n}{k}
\frac{(-1)^{k}q^{\binom{k}{2}+sk}H^q_{k-1}({\bf t})}{\qq{k}^{s}}.
\end{align*}
If $s=1$ then by the inductive step and \eqref{TT}, we get
\begin{align*}
\sum_{k=1}^n \qbin{n}{k}(-1)^{k}
q^{\binom{k}{2}}H^q_{k-1}({\bf t},1)
&=\frac{q^n }{\qq{n}}(-1)^{l({\bf t})+1}T^q_{n-1}({\bf t})
-(-1)^{l({\bf t})+1}T^q_{n}({\bf t},1)\\
&=(-1)^{l+1}\left(
T^q_{n}({\bf t},1)-\frac{q^nT^q_{n-1}({\bf t})}{\qq{n}}
\right)
=(-1)^{l+1} T^q_{n-1}({\bf s}).
\end{align*}
On the other hand, if $s>1$ then by the inductive step and \eqref{TT}, we obtain
\begin{align*}
\sum_{k=1}^n \qbin{n}{k}(-1)^{k}
q^{\binom{k}{2}}H^q_{k-1}({\bf t},s)
&=\frac{q^n}{\qq{n}}(-1)^{l({\bf t})+1}T^q_{n}({\bf t},s-1)
-(-1)^{l({\bf t})+1}T^q_{n}({\bf t},s)\\
&=(-1)^{l+1}\left(
T^q_{n}({\bf t},s)-\frac{q^nT^q_{n}({\bf t},s-1)}{\qq{n}}
\right)
=(-1)^{l+1} T^q_{n-1}({\bf s}),
\end{align*}
and therefore \eqref{I3} holds.

\noindent Putting
$a_n=H_{n-1}^q({\bf s})$ and
$b_n=(-1)^{l+1} T^q_{n-1}({\bf s})$
in \eqref{TC} and using \eqref{I3}, we immediately get
\begin{align*}
\sum_{k=1}^n \qbin{n}{k}\frac{(-1)^{k}
q^{\binom{k}{2}+kr}H^q_{k-1}({\bf s})}{\qq{k}^{r}}
&=\sum_{1\leq k_1\leq k_2\leq\cdots\leq k_r\le n}
\prod_{j=1}^r\frac{q^{k_j}}{\qq{k_j}}\cdot (-1)^{l+1} T^q_{k_1-1}({\bf s})\\
&=(-1)^{l+1} T^q_{n}({\bf s},r)
\end{align*}
and the proof of \eqref{I4} is complete.
Similarly, by \eqref{I3} and \eqref{TCD}, we obtain \eqref{I5}.
\end{proof}

\vspace{0.1cm}

\centerline{{\bf Acknowledgment.}}

\vspace{0.2cm}

 Kh.~Hessami Pilehrood and T.~Hessami Pilehrood acknowledge support from the Fields Institute Research Immersion Fellowships.

\medskip

\end{document}